\documentclass{amsart}
\usepackage[ocgcolorlinks,linktoc=all]{hyperref}
\usepackage{lipsum}

\hypersetup{citecolor=blue,linkcolor=red}
\newtheorem{theorem}{Theorem}[section]
\newtheorem*{thmA}{Theorem}

\newtheorem{lemma}[theorem]{Lemma}

\theoremstyle{definition}

\theoremstyle{remark}
\newtheorem{remark}[theorem]{Remark}

\numberwithin{equation}{section}

\begin{document}

\title[On Petty's conjectured projection inequality]
 {A local uniqueness theorem for minimizers of Petty's conjectured projection inequality}
\author[M.N. Ivaki]{Mohammad N. Ivaki}
\address{Institut f\"{u}r Diskrete Mathematik und Geometrie, Technische Universit\"{a}t Wien,
Wiedner Hauptstr. 8--10, 1040 Wien, Austria}
\email{mohammad.ivaki@tuwien.ac.at}

\dedicatory{}
\subjclass[2010]{}
\begin{abstract}
Employing the inverse function theorem on Banach spaces, we prove that in a $C^{2}(S^{n-1})$-neighborhood of the unit ball, the only solutions of $\Pi^2K=cK$ are origin-centered ellipsoids. Here $K$ is an $n$-dimensional convex body, $\Pi K$ is the projection body of $K$ and $\Pi^2K=\Pi(\Pi K).$
\end{abstract}

\maketitle
\section{introduction}
The setting of this paper is $n$-dimensional Euclidean space $\mathbb{R}^n$. A compact convex subset of $\mathbb{R}^{n}$ with non-empty interior is called a \emph{convex body}. The set of convex bodies in $\mathbb{R}^{n}$ is denoted by $\mathrm{K}^n$. Write $\mathrm{K}^n_e$ for the set of origin-symmetric convex bodies. Also, write $B^n$ for the unit ball of $\mathbb{R}^n$ and $S^{n-1}$ for the unit sphere of $\mathbb{R}^n$.

The support function of $K\in \mathrm{K}^n$, $h_K:S^{n-1}\to\mathbb{R}$, is defined by \[h_{K}(u):=\max_{x\in K}x\cdot u.\]
The space of support functions of bodies in $\mathrm{K}^n_e$ is denoted by $\mathrm{S}_e^n.$

Let $K$ be a convex body in $\mathbb{R}^n$, $n\geq 3.$ The projection body $\Pi K$ of $K$ is the origin-symmetric convex body whose support function is given by
\[h_{\Pi K}(u)=V_{n-1}(K|u^{\perp}),\]
for all $u\in S^{n-1}.$ Here, $V_{n-1}(K|u^{\perp})$ is the $(n-1)$-dimensional Lebesgue measure of the orthogonal projection of $K$ onto the subspace $u^{\perp}=\{x\in \mathbb{R}^n;~ x\cdot u=0\}$.

Petty's conjectured projection inequality states that for $n\ge3$ the minimum of the affine invariant quantity $\mathcal{P}(K):=V(\Pi K)/V(K)^{n-1}$ is attained precisely by ellipsoids; cf. \cite{Lut0,Lut1,Petty1,Petty}. Here $V(K)$ is the $n$-dimensional Lebesgue measure of $K.$ As Schneider \cite[p.~570]{Schneider} observes, if $K$ minimizes $\mathcal{P}$ then $\Pi^2K=cK+\vec{a}$ for some constant $c$ and vector $\vec{a}\in\mathbb{R}^n;$ therefore, the classification of solutions to $\Pi^2K=cK$ is of major interest; see also \cite[Problem 12.7]{Lut1} and \cite[Problems 4.5, 4.6]{Gardner}. Note that $\Pi B^n=\omega_{n-1}B^n$ and $\Pi^2B^n=\omega_{n-1}^nB^n,$ where $\omega_{k}$ is the volume of $B^{k}.$ Thus in view of $\Pi (\phi K)=|\det \phi|\phi^{-t}(\Pi K)$ for any $\phi\in \operatorname{Gl}_n$\footnote{The only $\operatorname{Sl}_n$-contravariant valuations on polytopes that are translation invariant are multiples of the projection operator; see \cite{Lud1,Lud2}.} (see, e.g., \cite[Theorem 4.1.5]{Gardner}), any origin-centered ellipsoid is also a solution. Weil \cite{W1} proved that the only polytopes that solve $\Pi^2K=cK+\vec{a}$ are the direct sums of centrally symmetric polygons or segments.  In view of Weil's result, to use the inverse function theorem on Banach spaces effectively, we restrict our attention to convex solutions with support functions of class $C^2$.
\begin{thmA}
Suppose $n\ge 3.$ There exists $\varepsilon>0$ with the following property. If $\Pi^2K=cK$ for some $c>0$, $h_{K}\in C^2(S^{n-1})$, and $\|h_{\phi K}-1\|_{C^{2}}\leq \varepsilon$ for some $\phi\in \operatorname{Gl}_n$, then
$K$ is an origin-centered ellipsoid.
\end{thmA}
Let us note that this result also follows from the recent work of Saroglou and Zvavitch \cite{SZ}, see also Remark \ref{keyrem}. It would be unfair if we do not mention the work of Fish, Nazarov, Ryabogin and Zvavich (FNRZ) \cite{FNRZ} where the idea of considering the iteration problems locally was initiated. In \cite{FNRZ}, FNRZ treated a similar question to the one considered here for iterations of the intersection body operator; the proof of Saroglou and Zvavitch is an adaption of the argument of \cite{FNRZ}. We believe that the techniques of \cite{FNRZ,SZ} apply to several other open problems and that our method here also applies to many other open problems such as the second mixed projection problem, the second mixed intersection problem and the projection centroid conjectures \cite{Gardner,Lut1} and yields local uniqueness theorems similar to the main theorem here. Due to notational complexity which would result in loss of clarity in exposition, we decided to treat these aforementioned problems elsewhere.

For the definitions of $W_{n-2}$ and $\Pi_{n-2}$ we refer the reader to (\ref{1}) and (\ref{2}).
In \cite{SZ}, the authors proved a strengthened version of Lutwak's inequality (cf. \cite{Lut}):
\begin{align}\label{SZ}
W_{n-2}(\Pi_{n-2}K)-\frac{n(n-2)\omega_{n-1}^2}{(n-1)^2\omega_n}W^2_{n-1}(K)-\frac{\omega_{n-1}^2}{(n-1)^2}W_{n-2}(K)\geq 0.
\end{align}
Furthermore, they conjectured that equality holds only for balls. In Section \ref{sec 5}, we give a direct proof of this inequality and we prove the equality holds precisely for convex bodies whose support functions lie in the linear span of spherical harmonics of degree less than or equal to two; see \cite[Theorem 5.7.4]{Groemer} for a geometric interpretation of such convex bodies.
\section*{Acknowledgment}
The work of the author was supported by Austrian Science Fund (FWF) Project M1716-N25 and the European Research Council (ERC) Project 306445. I would like to thank Saroglou and Zvavitch for a discussion related to Section \ref{sec 5}. I would like to thank immensely the referee for the very constructive suggestions that have led to a much better presentation of the original manuscript.
\section{Background}
In this section, we collect several standard definitions and facts from convex geometry. Most of the material here is taken from \cite{Gardner, Groemer, Schneider}.

A convex body is said to be of class $C^{2}_{+}$ if its boundary hypersurface is two times continuously differentiable, in the sense of differential geometry, and the Gauss map $\nu:\partial K\to S^{n-1}$, which takes $x$ on the boundary of $K$ to its unique outer unit normal vector $\nu(x)$, is well-defined and a $C^{1}$-diffeomorphism.
\subsection{Mixed volume, mixed surface area measure, and mixed projection}\label{basic}
Let $K,L$ be two convex bodies and $0<t<\infty$. The Minkowski sum $K+tL$ is defined by $h_{K+tL}:=h_K+th_L$ and the mixed volume $V_1(K,L)$ of $K$ and $L$ is defined by
\[V_1(K,L):=\frac{1}{n}\lim_{t\to0^{+}}\frac{V(K+tL)-V(K)}{t}.\]
A fundamental fact is that corresponding to each convex body $K$, there is a unique Borel measure $S(K,\cdot)$ on the unit sphere such that
\[V_1(K,L)=\frac{1}{n}\int\limits_{S^{n-1}}h_L(u)dS(K,u)\]
for any convex body $L$. The measure $S(K,\cdot)$ is called the surface area measure of $K.$ A convex body $K$ is said to have a positive continuous curvature function $f_K$, defined on the unit sphere, provided that for each convex body $L$
\[V_1(K,L)=\frac{1}{n}\int\limits_{S^{n-1}}h_Lf_Kdx,\]
where $dx$ is the spherical Lebesgue measure on $S^{n-1}.$ A convex body can have at most one curvature function; see \cite[p.~115]{bon}.
If $K$ is of class $C^2_+$, then $S_K$ is absolutely continuous with respect to $x$, and the Radon-Nikodym derivative $dS_K/dx:S^{n-1}\to\mathbb{R}$ is the reciprocal Gauss curvature of $\partial K$ (viewed as a function of the outer unit normal vectors) given by $f_K$. For every $K\in \mathrm{K}^{n},$ $V(K)=V_1(K,K).$

The mixed volume $V(K_1,\ldots,K_{n})$ and mixed area measure $S(K_1,\ldots,K_{n-1},\cdot)$ of the compact convex subsets $K_1,\ldots,K_{n}$ of $\mathbb{R}^n$ are respectively defined by
\begin{align*}
V(K_1,\ldots,K_{n})&:=\frac{1}{n!}\sum_{j=1}^n(-1)^{n+j}\sum_{i_1<\ldots<i_j}V(K_{i_1}+\cdots+K_{i_j}),\\
S(K_1,\ldots,K_{n-1},\cdot)&:=\frac{1}{(n-1)!}\sum_{j=1}^{n-1}(-1)^{n+j-1}\sum_{i_1<\ldots<i_j}S(K_{i_1}+\cdots+K_{i_j},\cdot).
\end{align*}
Furthermore, they are related by
\[V(K_1,\ldots,K_{n})=\frac{1}{n}\int\limits_{S^{n-1}}h_{K_n}(u)dS(K_{1},\ldots,K_{n-1},u).\]

For real symmetric $(n-1)\times(n-1)$ matrices $A_1,\ldots,A_{n-1}$, write $D(A_1,\ldots,A_{n-1})$ for their mixed discriminant; see \cite[(2.64), (5.117)]{Schneider}. Let $S_{n-1}$ be the group of all permutations of the set $\{1,\ldots,n-1\}$ and let $\varepsilon:S_{n-1}\to \{-1,1\}$ be defined by $\varepsilon(\sigma)=1$ $(-1)$ if $\sigma$ is even (odd). The mixed discriminant of functions $f_i\in C^{2}(S^{n-1}), 1\leq i\leq n-1,$ is a multilinear operator defined as
\begin{align}\label{def q}
\mathcal{Q}(f_1,\ldots,f_{n-1})&:=D\left(A[f_1],\ldots, A[f_{n-1}]\right)\\
&=\frac{1}{(n-1)!}\sum_{\delta,\tau\in S_{n-1}}\varepsilon(\delta\tau)\prod_{i=1}^{n-1}(A[f_i])_{\delta(i)\tau(i)},\nonumber
\end{align}
where in a local orthonormal frame of $S^{n-1}$ the entries of the matrix $A[f_k]$ are given by $(A[f_k])_{ij}=\nabla_i\nabla_jf_k+\delta_{ij}f_k$ and $\nabla$ is the covariant derivative on $S^{n-1}.$ The operator $\mathcal{Q}$ enjoys several important properties, for example, it is independent of the order of its arguments; see \cite[Lemma 2-12]{An}.
\begin{remark}
Suppose $f,g\in C^2(S^{n-1}).$ For convenience we will put
\begin{align*}
\Box f:=\Delta f+(n-1)f,~\mathcal{Q}(f):=\mathcal{Q}(f,\ldots,f),~ \mathcal{Q}_1(f,g):=\mathcal{Q}(f,\ldots,f,g).
\end{align*}
Note that $\mathcal{Q}(f)=\det (A[f])$ and a simple calculation shows that $\mathcal{Q}_1(1,f)=\frac{1}{n-1}\Box f$.
\end{remark}
Let $K$ and $K_i$ be $C^2_+$ convex bodies. Using \cite[(2.68), (5.48)]{Schneider}, for any Borel set $\omega\subset S^{n-1}$, we have
\begin{align*}
S(K_1,\ldots,K_{n-1},\omega)&=\int\limits_{\omega}\mathcal{Q}(h_{K_1},\ldots,h_{K_{n-1}})(x)dx,\\
S(B^n,\dots,B^n,K,\omega)&=\frac{1}{n-1}\int\limits_{\omega}\Box h_K(x)dx.
\end{align*}
The mixed projection $\Pi(K_1,\ldots,K_{n-1})$ of convex bodies $K_1,\ldots,K_{n-1}$ is a convex body whose support function is given by
\begin{align*}
h_{\Pi(K_1,\ldots,K_{n-1})}(u)&:=V_{n-1}(K_1|u^{\perp},\ldots,K_{n-1}|u^{\perp})\\
&=\frac{1}{2}\int\limits_{S^{n-1}}|u\cdot v|dS(K_1,\ldots,K_{n-1},v)\\
&=nV(K_1,\ldots,K_{n-1}, \bar{u}),
\end{align*}
where $\bar{u}$ is the segment joining $-u/2$ and $u/2$; see \cite[page 570]{Schneider}.
\begin{remark}
For simplicity, we write
\[\Pi(K,\ldots,K,L)=\Pi_1(K,L).\]
\end{remark}
\subsection{Spherical Harmonics}
Write $L^2(S^{n-1})$ for the Hilbert space of square-integrable real functions on $S^{n-1}$ equipped with the scalar product
\[(f,g):=\int\limits_{S^{n-1}}fgdx.\]
The induced norm by this scalar product is symbolized by $\|\cdot\|_2.$

Spherical harmonics of degree $k$ are eigenfunctions of the spherical Laplace operator $\Delta$ with the eigenvalue $k(k+n-2).$ In fact, if $Y_k$ is such function then
\[\Delta Y_k=-k(k+n-2)Y_k.\]
The set $\mathcal{S}^k$ of spherical harmonics of degree $k$ is a vector subspace of $C(S^{n-1}).$ Moreover,
$\operatorname{dim}\mathcal{S}^k=N(n,k):=\frac{2k+n-2}{k+n-2}\binom{k+n-2}{k}
.$ In each space $\mathcal{S}^k,$ choose an orthonormal basis
$\{Y_{k,1},\ldots,Y_{k,N(n,k)}\}$. Note that
\[Y_{0,1}=\frac{1}{\sqrt{n\omega_n}}.\]
For $f\in L^2(S^{n-1})$ we write
\[\pi_kf:=\sum_{l=1}^{N(n,k)}\left(f,Y_{k,l}\right)Y_{k,l},\quad\pi_0f=\frac{1}{n\omega_n}\int\limits_{S^{n-1}}fdx.\]
The condensed harmonic expansion of $f$ is given by
\[f\sim \sum_{k=0}^{\infty}\pi_kf;\]
it converges to $f$ in the $L^2(S^{n-1})$-norm. In addition, for $f,g\in L^2(S^{n-1})$ we have
\[\sum_{k=0}^{\infty}\sum_{l=1}^{N(n,k)}\left(f,Y_{k,l}\right)\left(g,Y_{k,l}\right)=\left(f,g\right).\]
Note that $f\in L^2(S^{n-1})$ if and only if its condensed harmonic expansion satisfies
\[\sum_{k=0}^{\infty}\|\pi_kf\|_2^2<\infty.\]
One can read more about spherical harmonics in \cite{Groemer}.
\subsection{Radon transform and cosine transform}
Suppose that $f$ is a Borel function on $S^{n-1}.$ The spherical Radon transform (also
known as the Funk Transform; see, for example, \cite{helgason}) and cosine transform of $f$ are defined as follows
\begin{align*}
\mathcal{R}f(u):=\frac{1}{(n-1)\omega_{n-1}}\int\limits_{S^{n-1}\cap u^{\perp}}f(x)dx,\quad
\mathcal{C}f(u):=\int\limits_{S^{n-1}}|u\cdot x|f(x)dx.
\end{align*}
We normalized $\mathcal{R}$ so that $\mathcal{R}1=1.$
The transformations $\mathcal{R}$ and $\mathcal{C}$ are self-adjoint, in the sense that if $f$ and $g$ are bounded Borel functions on $S^{n-1},$ then
\[\int\limits_{S^{n-1}}f(x)\mathcal{R}g(x)dx=\int\limits_{S^{n-1}}g(x)\mathcal{R}f(x)dx,~\int\limits_{S^{n-1}}f(x)\mathcal{C}g(x)dx=
\int\limits_{S^{n-1}}g(x)\mathcal{C}f(x)dx.\]
For a continuous function $f\in C(S^{n-1})$ the explicit expressions for the first and second derivatives of $\mathcal{C}f$ are given in \cite[Theorem 1, Lemma 1]{Yves}: Let $f$ be a continuous function on $S^{n-1}.$ Then the positively 1-homogeneous function
\begin{align*}
h: \mathbb{R}^n\to \mathbb{R},\quad p\mapsto \int\limits_{S^{n-1}}|p\cdot x|f(x)dx,
\end{align*}
is of class $C^2$ on $\mathbb{R}^n-\{0\}.$ Its first differential and second differential at $p,$ considered as a bilinear form on $\mathbb{R}^n,$ are given by
\begin{align}\label{second1}
d_ph=\int\limits_{S^{n-1}}\operatorname{sgn}(x\cdot p)xf(x)dx,\quad d_p^2h(x,y)=\frac{2}{\|p\|}\int\limits_{S^{n-1}\cap p^{\perp}}(q\cdot x)(q\cdot y)f(q)dq.
\end{align}
for all $x,y \in\mathbb{R}^n.$

Identities (\ref{second1}) imply that if for a sequence $\{f_i\}_i\subset C(S^{n-1})$, $f_i\to_{C} f\in C(S^{n-1})$ then $\mathcal{C}f_i\to_{C^2} \mathcal{C}f\in C^2(S^{n-1}).$
Thus $\mathcal{C}:C(S^{n-1})\to C^2(S^{n-1})$ is a continuous linear map from a Banach space to a Banach space; therefore, it is bounded. That is, there exists a constant $c_n>0$ such that for any $f\in C(S^{n-1})$ there holds
\begin{align}\label{second}
\|\mathcal{C}f\|_{C^2}\leq c_n \|f\|_{C}.
\end{align}
The following relation between Radon transform and cosine transform is established in \cite[Proposition 2.1]{Goodey}:
\begin{align}\label{C and R}
\Box\mathcal{C}=2(n-1)\omega_{n-1}\mathcal{R}.
\end{align}
Let $H^s(S^{n-1}),$ $s\ge 0$,  be the spaces of those functions for which the spherical harmonic expansion satisfies
$
\|f\|_{H^s}^2:=\sum_{k=0}^{\infty}(1+k^2)^s\|\pi_kf\|_2^2<\infty.
$
The following results about the smoothing property of $\mathcal{R}, ~\mathcal{C}$ are proved in \cite{STRICHARTZ}:
\begin{align}\label{STRICHARTZ R}
\|\mathcal{R}f\|_{H^{s+\frac{n-2}{2}}}\leq a_{s,n}\|f\|_{H^s},\quad
\|\mathcal{C}f\|_{H^{s+\frac{n+2}{2}}}\leq b_{s,n}\|f\|_{H^s}
\end{align}
for some positive constants depending on $s$ and $n.$
Let us put $$H_e^s(S^{n-1}):=\{f\in H^s(S^{n-1});~ f(x)=f(-x),~ \forall x\in S^{n-1}\}.$$ Strichartz proved that $H_e^s(S^{n-1})$ is precisely the space of even functions $f\in L^2(S^{n-1})$ with derivatives up to order $s$ in $L^2(S^{n-1});$ see \cite[Pages 721-722]{STRICHARTZ}. Hence $H_e^s(S^{n-1})$, $s\ge0$, are Sobolev spaces.
\section{a local diffeomorphism}
Define a map by
\begin{align*}
\mathcal{X}_m: \mathrm{S}_e^n&\to C_e(S^{n-1}),\quad
h_K\mapsto-h_{\Pi^{2m}K}+\left(\frac{V(\Pi^{2m}K)}{V(K)}\right)^{\frac{1}{n}}h_K.
\end{align*}
Clearly $\mathcal{X}_m$ is a continuous nonlinear map.
\begin{remark}\label{rem: key}
Let $\tilde{U}$ be a $C^2_e(S^{n-1})$-neighborhood of $0$ such that for every $f\in \tilde{U},$ $1+f>0$ and the matrix $A[1+f]$ is positive definite. Thus for each $f\in \tilde{U},$ $1+f$ represents the support function of an origin-symmetric convex body of class $C^{2}_+.$ Also note that if $K$ is of class $C^2_+$, then $\Pi^{k}K$ is also of class $C^2_+;$ see \cite[p. 13]{Yves}. Putting these two facts together implies that $\mathcal{X}_m(1+\cdot)$ maps $\tilde{U}$ to a subset of $C_e^2(S^{n-1}).$
\end{remark}
To employ the inverse function theorem, it is convenient to work with the map
\begin{align*}
\mathcal{Y}_m&:\tilde{U}\subset C_e^2(S^{n-1})\to C_e^2(S^{n-1})\\
\mathcal{Y}_m(f)&:=\mathcal{X}_m(1+f).
\end{align*}
The derivative of $\mathcal{Y}_m:\tilde{U}\subset C_e^2(S^{n-1})\to C_e^2(S^{n-1})$ at the point $f\in \tilde{U}$ in the direction $g\in C_e^2(S^{n-1})$ is defined by
\begin{equation}\label{derv def}
\lim_{t\to0}\left\|\frac{\mathcal{Y}_m(f+tg)-\mathcal{Y}_m(f)}{t}-D\mathcal{Y}_m(f,g)\right\|_{C^2}=0.
\end{equation}
Similar definitions apply to the higher derivatives. The $k+1$-th derivative is defined by induction:
\begin{align*}
D^{k+1}\mathcal{Y}_m&:\tilde{U}\times \left(C_e^2(S^{n-1})\right)^{(k+1)}\to C_e^2(S^{n-1})\\
D^{k+1}\mathcal{Y}_m(f,g_1,\ldots,g_{k+1})&=\lim_{t\to0}\frac{D^{k}\mathcal{Y}_m(f+tg_{k+1},g_1,\ldots,g_{k})-D^{k}\mathcal{Y}_m(f,g_1,\ldots,g_{k})}{t}.
\end{align*}
The limit is taken in the $C^2$-norm. We prove that $\mathcal{Y}_m$ is smooth; that is,
\begin{enumerate}
  \item $\mathcal{Y}_m$ is continuous,
  \item the limits above exist for all $k\ge 0,$ $f\in \tilde{U}$ and \[(g_1,\ldots,g_{k+1})\in \left(C_e^2(S^{n-1})\right)^{(k+1)}=\underbrace{C_e^2(S^{n-1})\times\cdots\times C_e^2(S^{n-1})}_{k+1~\text{times}},\]
  \item $D^{k+1}\mathcal{Y}_m$ is continuous jointly as a function on the product space.
\end{enumerate}
There is a large difference in what it means for $D\mathcal{Y}_m: \tilde{U}\times C^2_e(S^{n-1})\to C^2_e(S^{n-1})$ to be continuous, as opposed to $D\mathcal{Y}_m: \tilde{U}\subset C^2_e(S^{n-1})\to L(C^2_e(S^{n-1}),C^2_e(S^{n-1}))$ (thus our continuity assumption is weaker and easier to check); see \cite[Definition 3.1.1]{Hamilton}.

Then we continue by providing the explicit expression of $D\mathcal{Y}_m(0,\cdot)$ and a description of $\operatorname{Ker}D\mathcal{Y}_m(0,\cdot).$ The description of the kernel yields that
$$D\mathcal{Y}_m(0,\cdot)+\pi_0+\pi_2:C_e^2(S^{n-1})\to C_e^2(S^{n-1})$$ is an isomorphism, so by the inverse function theorem (see, \cite[Theorem 5.2.3]{Hamilton}) there are open neighborhoods $U,~W$ of $0$ in $C^2_e(S^{n-1})$, such that the map
\begin{align*}
\mathcal{N}:U\subset \tilde{U}&\to W\subset C_e^2(S^{n-1}),\quad f\mapsto\mathcal{Y}_m(f)+(\pi_0+\pi_2)f
\end{align*}
is a smooth diffeomorphism.
\begin{lemma}\label{lem 4} For $K\in\mathrm{K}^n_e$ of class $C^2_+$ and $g\in C_e^2(S^{n-1})$ we have
\begin{align*}
\frac{d}{dt}\Big|_{t=0}h_{\Pi^{k}(h_K+tg)}=
\frac{(n-1)^k}{2^k}\mathcal{C}\mathcal{Q}_1(h_{\Pi^{k-1}K},\mathcal{C}\mathcal{Q}_1(h_{\Pi^{k-2}K},\mathcal{C}\mathcal{Q}_1(\ldots,\mathcal{C}\mathcal{Q}_1(h_K,g)\cdots))),
\end{align*}
$$\frac{d}{dt}\Big|_{t=0}h_{\Pi^{k}(1+tg)}=\frac{\left(\frac{V(\Pi^{k}B^n)}{V(B^n)}\right)^{\frac{1}{n}}}{2^k\omega_{n-1}^k}(\mathcal{C}\Box)^{(k)} g,$$
where $(\mathcal{C}\Box)^{(k)} g:=\underbrace{\mathcal{C}\Box\ldots\mathcal{C}\Box}_{k~\text{times}} g.$
\end{lemma}
\begin{proof}
In this proof, $f=g+o(t^l,k)$ means that $\|f-g\|_{C^k}\leq ct^l$ for some $c>0.$
Take $t$ sufficiently small enough so that $h_K+tg$ is the support function of a $C^2_+$ convex body which we denote also by $h_K+tg$. By (\ref{def q}), we have
\[\mathcal{Q}(h_K+tg,\ldots,h_K+tg)=\mathcal{Q}(h_K)+(n-1)t\mathcal{Q}_1(h_K,g)+o(t^2,0).\]
Hence in view of the definition of mixed projection and (\ref{second}) we get
\begin{align*}
h_{\Pi(h_K+tg)}&=h_{\Pi K}+\frac{t(n-1)}{2}\mathcal{C}\mathcal{Q}_1(h_K,g)+o(t^2,2).
\end{align*}
If $t$ is small enough, then $h_{\Pi(h_K+tg)}$ also represents a $C^2_+$ convex body, so
\begin{align*}
h_{\Pi^2(h_K+tg)}&=h_{\Pi^2K}+\frac{t(n-1)^2}{4}\mathcal{C}\mathcal{Q}_1(h_{\Pi K},\mathcal{C}\mathcal{Q}_1(h_K,g))+o(t^2,2).
\end{align*}
By induction and (\ref{second}), for $t$ small enough, we obtain
\begin{align}\label{som}
&\frac{h_{\Pi^{k}(h_K+tg)}-h_{\Pi^{k}K}}{t}\\
&=
\frac{(n-1)^k}{2^k}\mathcal{C}\mathcal{Q}_1(h_{\Pi^{k-1}K},\mathcal{C}\mathcal{Q}_1(h_{\Pi^{k-2}K},\mathcal{C}\mathcal{Q}_1(\ldots,\mathcal{C}\mathcal{Q}_1(h_K,g)\cdots)))+o(t,2)\nonumber.
\end{align}
This proves the first assertion. For $K=B^n,$ (\ref{som}) yields
\begin{align*}
h_{\Pi^{k}(1+tg)}&=h_{\Pi^{k}B^n}+\frac{t}{2^{k}}
\left(\prod_{i=1}^k c_i\right)^{n-2}(\mathcal{C}\Box)^{(k)} g+o(t^2,2),
\end{align*}
where $c_i$ are defined through $c_1=1$ and $\Pi^{i-1}B^n=c_iB^n.$ So we need to verify that
\[\left(\prod_{i=1}^k c_i\right)^{n-2}=\frac{\left(\frac{V(\Pi^{k}B^n)}{V(B^n)}\right)^{\frac{1}{n}}}{\omega_{n-1}^k}.\]
We prove the claim by induction. The claim clearly holds for $k=1.$ Suppose for some $k\ge 1,$
$\left(\prod_{i=1}^k c_i\right)^{n-2}=\frac{\left(V(\Pi^{k}B^n)/V(B^n)\right)^{1/n}}{\omega_{n-1}^k}.$
Hence we see that
\[\left(\prod_{i=1}^{k+1} c_i\right)^{n-2}=\frac{\left(\frac{V(\Pi^{k}B^n)}{V(B^n)}\right)^{\frac{1}{n}}}{\omega_{n-1}^k}c_{k+1}^{n-2}.\]
Since $\Pi^{k}B^n=c_{k+1}B^n,$ we get $\left(\prod_{i=1}^{k+1} c_i\right)^{n-2}=\frac{c_{k+1}^{n-1}}{\omega_{n-1}^k}.$ Also, we have $c_{k+2}=c_{k+1}^{n-1}\omega_{n-1};$ therefore, we arrive at
\[\left(\prod_{i=1}^{k+1} c_i\right)^{n-2}=\frac{c_{k+2}}{\omega_{n-1}^{k+1}}\Rightarrow \left(\prod_{i=1}^{k+1} c_i\right)^{n-2}=\frac{\left(\frac{V(\Pi^{k+1}B^n)}{V(B^n)}\right)^{\frac{1}{n}}}{\omega_{n-1}^{k+1}}.\]
\end{proof}
\begin{remark}The following remarks are in order.
\begin{enumerate}
\item In Lemma \ref{lem 4}, since $g\in C_e^2(S^{n-1})$, we have $(\mathcal{C}\Box)^{(k)}g \in C_e^2(S^{n-1})$; see \cite{Yves}.
\item Let us put $\alpha_{\Pi^kK}(g):=\frac{d}{dt}\Big|_{t=0}h_{\Pi^{k}(h_K+tg)},~ \beta_{\Pi^kK}(g):=\frac{d}{dt}\Big|_{t=0}f_{\Pi^{k}(h_K+tg)}.$ Using Lemma \ref{lem 4}, we have
\begin{align}
\alpha_{\Pi K}(g)&=\frac{n-1}{2}\mathcal{C}\mathcal{Q}_1(h_{K},g),\nonumber\\
\alpha_{\Pi^kK}(g)&=\frac{n-1}{2}\mathcal{C}\mathcal{Q}_1(h_{\Pi^{k-1}K},\alpha_{\Pi^{k-1}K}(g)),\label{derv induc}\\
\beta_{\Pi^kK}(g)&=(n-1)\mathcal{Q}_1(h_{\Pi^kK},\alpha_{\Pi^k K}(g)).\nonumber
\end{align}
\end{enumerate}
\end{remark}
\begin{lemma}\label{lemma 5} For $K\in\mathrm{K}^n_e$ of class $C^2_+$ and $g\in C_e^2(S^{n-1})$  we have
\begin{align*}
&\frac{d}{dt}\Big|_{t=0}V(\Pi^{k}(h_K+tg))\\
&=\frac{(n-1)^{k}}{2^{k-1}}\int_{S^{n-1}}g\mathcal{Q}_1(h_K,\mathcal{C}\mathcal{Q}_1(h_{\Pi K},
\mathcal{C}\mathcal{Q}_1(\ldots,\mathcal{C}\mathcal{Q}_1(h_{\Pi^{k-1}K},h_{\Pi^{k+1}K})\cdots)))dx,
\end{align*}
\[\frac{d}{dt}\Big|_{t=0}V(\Pi^{k}(1+tg))=\frac{\left(\frac{V(\Pi^{k}B^n)}{V(B^n)}\right)^{\frac{1}{n}}}{2^k\omega_{n-1}^k} \left((\Box\mathcal{C})^{(k)}f_{\Pi^{k}B^n}\right)\int_{S^{n-1}}gdx.\]
\end{lemma}
\begin{proof}
Let $L$ be a convex body of class $C^2_+.$ By \cite[Lemma 2-12]{An}, for $f\in C_e^2(S^{n-1})$ we have
\begin{align}\label{xxx}
\int_{S^{n-1}}f\mathcal{Q}_1(h_L,g)dx=\int_{S^{n-1}}g\mathcal{Q}_1(h_L,f)dx.
\end{align}
Therefore, we obtain
\[\frac{d}{dt}\Big|_{t=0}V(\Pi^{k}(h_K+tg))=
\int_{S^{n-1}}\alpha_{\Pi^{k}K}(g)f_{\Pi^{k}K}dx.\]
Note that $\mathcal{C}f_{\Pi^{k}K}=2h_{\Pi^{k+1}K}$. Using (\ref{derv induc}), $\mathcal{C}$ is self-adjoint and (\ref{xxx}), we calculate
\begin{align*}
&\int_{S^{n-1}}\alpha_{\Pi^{k}K}(g)f_{\Pi^{k}K}dx\\
&=(n-1)\int_{S^{n-1}}\mathcal{Q}_1(h_{\Pi^{k-1}K},\alpha_{\Pi^{k-1}K}(g))h_{\Pi^{k+1}K}dx \\
&=(n-1)\int_{S^{n-1}}\alpha_{\Pi^{k-1}K}(g)\mathcal{Q}_1(h_{\Pi^{k-1}K},h_{\Pi^{k+1}K})dx \\
&=\frac{(n-1)^2}{2}\int_{S^{n-1}}\mathcal{Q}_1(h_{\Pi^{k-2}K},\alpha_{\Pi^{k-2}K}(g))\mathcal{C}\mathcal{Q}_1(h_{\Pi^{k-1}K},h_{\Pi^{k+1}K})dx \\
&=\frac{(n-1)^{k}}{2^{k-1}}\int_{S^{n-1}}g\mathcal{Q}_1(h_K,\mathcal{C}\mathcal{Q}_1(h_{\Pi K},
\mathcal{C}\mathcal{Q}_1(\ldots,\mathcal{C}\mathcal{Q}_1(h_{\Pi^{k-1}K},h_{\Pi^{k+1}K})\cdots)))dx.
\end{align*}
The second claim follows similarly.
\end{proof}
\begin{lemma}\label{cont key}
The map $\mathcal{Y}_m:\tilde{U}\to C_e^2(S^{n-1})$ is smooth.
\end{lemma}
\begin{proof}
Lemmas \ref{lem 4} and \ref{lemma 5} prove that the limit in (\ref{derv def}) exists  and it is linear in the second argument.
We prove that $D\mathcal{Y}_m: \tilde{U}\times C_e^2(S^{n-1})\to C_e^2(S^{n-1})$ is continuous (the following argument also shows the continuity of $\mathcal{Y}_m$).
Suppose $f_i\to_{C^2} f\in \tilde{U} $ and $g_i\to_{C^2}g\in C_e^2(S^{n-1}).$ There are convex bodies $\{K,K_i\}$ of class $C^2_+$ such that $1+f_i=h_{K_i}$ and $1+f=h_K.$
Since $h_{K_i}\to_{C^2} h_{K},$  we have $f_{K_i}\to_C f_{K}$ and so $$2h_{\Pi K_i}=\mathcal{C}f_{K_i}\to_{C^2} \mathcal{C}f_{K}=2h_{\Pi K}.$$
By induction, $h_{\Pi^lh_{K_i}}\to_{C^2}h_{\Pi^lh_{K}}$ for all $1\leq l\leq 2m+1.$ The continuity now follows from
$D\mathcal{Y}_m(f_i,g_i)=\frac{d}{dt}\Big|_{t=0}\mathcal{X}_m(h_{K_i}+tg_i)$ and Lemmas \ref{lem 4} and \ref{lemma 5}.

Note that $D^2\mathcal{X}_m(h_{K},g_1,g_2)=\frac{d}{dt}\Big|_{t=0}D\mathcal{X}_m(h_{K}+tg_2,g_1).$ Since sufficient conditions for interchanging of differentiation and integration are satisfied, by (\ref{second}) and (\ref{derv induc}), we conclude that $D^2\mathcal{X}_m$ exists and is continuous. The existence and continuity of higher derivatives follow by induction.
\end{proof}
\begin{remark}
Our argument above shows that in fact for any $K$ of class $C^2_+$, there is a $C^2$-neighborhood of $h_K$ such that $\mathcal{X}_m$ is smooth.
\end{remark}
The next lemma provides an interesting expression for $D\mathcal{Y}_m(0,\cdot)$ in terms of the spherical Radon transform.
\begin{lemma}
For any $g\in C^2_e(S^{n-1})$ we have
\[D\mathcal{Y}_m(0,g)=\left(\frac{V(\Pi^{2m}B^n)}{V(B^n)}\right)^{\frac{1}{n}}\left(g-(n-1)^{2m}\mathcal{R}^{2m} g+\frac{(n-1)^{2m}-1}{n\omega_n}\int\limits_{S^{n-1}} gdx\right).\]
\end{lemma}
\begin{proof}
Using Lemmas \ref{lem 4} and \ref{lemma 5}, we calculate
\begin{align}\label{time der}
\frac{d}{dt}\Big|_{t=0}\mathcal{Y}_m(tg)=&-\frac{\left(\frac{V(\Pi^{2m}B^n)}{V(B^n)}\right)^{\frac{1}{n}}}{2^{2m}\omega_{n-1}^{2m}}(\mathcal{C}\Box)^{(2m)} g+\left(\frac{V(\Pi^{2m}B^n)}{V(B^n)}\right)^{\frac{1}{n}}g\nonumber\\
&+\frac{\left(\frac{V(\Pi^{2m}B^n)}{V(B^n)}\right)^{\frac{2}{n}-1}}{2^{2m}\omega_{n-1}^{2m}} \left((\Box\mathcal{C})^{(2m)}f_{\Pi^{2m}B^n}\right)\frac{\int\limits_{S^{n-1}}gdx}{n\omega_n}\\
&-\left(\frac{V(\Pi^{2m}B^n)}{V(B^n)}\right)^{\frac{1}{n}}\frac{1}{n\omega_n}\int\limits_{S^{n-1}}gdx.\nonumber
\end{align}
 Since $\mathcal{Y}_m(t)=0$ for $t$ small enough,
\begin{align*}
\frac{d}{dt}\Big|_{t=0}\mathcal{Y}_m(t)=0\Rightarrow-(\mathcal{C}\Box)^{(2m)}1+\left(\frac{V(\Pi^{2m}B^n)}{V(B^n)}\right)^{\frac{1}{n}-1}
\left((\Box\mathcal{C})^{(2m)}f_{\Pi^{2m}B^n}\right)=0.
\end{align*}
Rearranging the terms yields
\begin{align*}
\left(\frac{V(\Pi^{2m}B^n)}{V(B^n)}\right)^{\frac{1}{n}-1}\left((\Box\mathcal{C})^{(2m)}f_{\Pi^{2m}B^n}\right)=(2\omega_{n-1}(n-1))^{2m}.
\end{align*}
Substituting this last identity into (\ref{time der}) yields
\[D\mathcal{Y}_m(0,g)=\left(\frac{V(\Pi^{2m}B^n)}{V(B^n)}\right)^{\frac{1}{n}}\left(g-\frac{(\mathcal{C}\Box)^{(2m)} g}{2^{2m}\omega_{n-1}^{2m}}+\frac{(n-1)^{2m}-1}{n\omega_n}\int\limits_{S^{n-1}} gdx\right).\]
For $g\in C^2(S^{n-1})$, $\mathcal{C}\Box g=\Box\mathcal{C}g$, so the identity (\ref{C and R}) finishes the proof.
\end{proof}
\begin{lemma}\label{lem: dim ker}
$\operatorname{dim}\operatorname{Ker}D\mathcal{Y}_{m}(0,\cdot)=\frac{n(n+1)}{2}.$
\end{lemma}
\begin{proof}
Suppose $g\in C^2_e(S^{n-1})$. Recall from \cite[Lemma 3.4.7]{Groemer} that $$\mathcal{R}\pi_kg=(-1)^{\frac{k}{2}}v_{k,n}\pi_kg,$$ where
\[v_{k,n}=\left\{
    \begin{array}{ll}
      \frac{1\cdot3\ldots(k-1)}{(n-1)(n+1)\ldots(n+k-3)} & \hbox{$k$~\text{even};} \\
      0 & \hbox{$k$~\text{odd}.}
    \end{array}
  \right.
\]
Since $D\mathcal{Y}_{m}(0,1)=0$, we restrict our attention to the space of even functions with $\pi_0g=\frac{1}{n\omega_n}\int_{S^{n-1}} gdx =0.$ Note that $\|g\|_2^2=\sum_{k=1}^{\infty}\|\pi_{2k}g\|_2^2$ and $\mathcal{R}\pi_kg=\pi_k\mathcal{R}g$.\footnote{$\mathcal{R}\pi_kg=(-1)^{\frac{k}{2}}v_{k,n}\pi_kg=\sum\limits_l\int\limits_{S^{n-1}}g\mathcal{R}Y_{k,l}dxY_{k,l}
=\sum\limits_l\int\limits_{S^{n-1}}\mathcal{R}gY_{k,l}dxY_{k,l}=\pi_k\mathcal{R}g.$} For all $j\in \mathbb{N},$ we have
  \[\sum_{k=1}^{j}\pi_{2k}g-(n-1)^{2m}\mathcal{R}^{2m}(\sum_{k=1}^{j}\pi_{2k}g)=\sum_{k=1}^j\left(1-(n-1)^{2m}v_{2k,n}^{2m}\right)\pi_{2k}g,\]
  \[\|g-(n-1)^2\mathcal{R}^{2m}g\|_2^2=\sum_{k=1}^{\infty}\left(1-(n-1)^{2m}v_{2k,n}^{2m}\right)^2\|\pi_{2k}g\|_2^2.\]
Recall that $v_{2,n}=\frac{1}{n-1}.$ Hence $g-(n-1)^{2m}\mathcal{R}^{2m}g=0$ if and only if $\pi_{2k}g=0$ for all $k\ne 0,1.$ Consequently,  $\operatorname{Ker}\mathcal{Y}_{m}(0,\cdot)=\mathcal{S}^0\oplus\mathcal{S}^2,$ which is of dimension $N(n,0)+N(n,2)=\frac{n(n+1)}{2}.$
\end{proof}
\begin{lemma}\label{lem 9}
Suppose $m\ge4.$ Given $h\in C^{2}_e(S^{n-1}) $ with $\pi_kh=0$ for $k= 0, 2$, there exists a unique $g\in C_e^{2}(S^{n-1}) $ with $\pi_kg=0$ for $k= 0, 2$ such that
$$g-(n-1)^{2m}\mathcal{R}^{2m} g=h.$$
\end{lemma}
\begin{proof}
We develop $h$ into a series of spherical harmonics:
$h\sim \sum_{k\ne 0,1}^{\infty} \pi_{2k}h.$
Since $L^2(S^{n-1})$ is a complete space and $\lim\limits_{k\to\infty }1-(n-1)^{2m}v_{2k,n}^{2m}= 1$, the $L^2(S^{n-1})$-Cauchy sequence
\begin{align*}
\left\{f_l:=\sum_{k\ne 0,1}^{l} \frac{1}{1-(n-1)^{2m}v_{2k,n}^{2m}}\pi_{2k}h\right\}_l
\end{align*}
converges in the $L^2(S^{n-1})$-norm to a bounded even $f\in L^2(S^{n-1})\cap \left(\mathcal{S}^0\oplus\mathcal{S}^2\right)^{\perp}$ with
$$\pi_{2k}f=\frac{1}{1-(n-1)^{2m}v_{2k,n}^{2m}}\pi_{2k}h$$ for $k\ge 2.$ In view of (\ref{STRICHARTZ R}), $\mathcal{R}^{2m}f\in H_e^{m(n-2)}\subset H_e^{4(n-2)}\subset C^{2}(S^{n-1}).$
Define $$g:=h+(n-1)^{2m}\mathcal{R}^{2m}f.$$ Note that $g\in C^2_e(S^{n-1})\cap \left(\mathcal{S}^0\oplus\mathcal{S}^2\right)^{\perp}$ and for $k\ge 2:$
\[\pi_{2k}g=\left(1+\frac{(n-1)^{2m}v_{2k,n}^{2m}}{1-(n-1)^{2m}v_{2k,n}^{2m}}\right)\pi_{2k}h\Rightarrow\pi_{2k}(g-(n-1)^{2m}\mathcal{R}^{2m}g)=\pi_{2k}h.\]
Since $h$ and $g-(n-1)^{2m}\mathcal{R}^{2m}g$ are $C^2,$
$g-(n-1)^{2m}\mathcal{R}^{2m}g=h.$

\noindent The proof of the uniqueness is elementary.
\end{proof}
\section{proof of the main result}\label{sec4}
In this section, $E$ always represents an ellipsoid
and a vector $a\in \mathbb{R}^{\frac{n(n+3)}{2}}$ is written as $a=(a_0,a_{1,1},\ldots, a_{1,n},a_{2,1},\ldots, a_{2,N(n,2)}).$

Let $\mathcal{O}_1$ be an open ball about the origin in $\mathbb{R}^{\frac{n(n+3)}{2}}$ such that for any $a\in \mathcal{O}_1:$
\begin{enumerate}
  \item $a_0+\sqrt{n}>0,$
  \item $\frac{(a_0+\sqrt{n})^2}{n}+\sum\limits_{i=1}^{N(n,2)}a_{2,i}Y_{2,i}>0,$
  \item $s_a:=\sum_{i=1}^na_{1,i}Y_{1,i}+\left(\frac{(a_0+\sqrt{n})^2}{n}+\sum\limits_{i=1}^{N(n,2)}a_{2,i}Y_{2,i}\right)^{\frac{1}{2}}>0,$
  \item  $s_a$ is the support function of a convex body of class $C^{2}_+$.
\end{enumerate}
The function $\frac{n}{(a_0+\sqrt{n})^2}(s_a-\sum_{i=1}^na_{1,i}Y_{1,i})^2$ is the restriction of a positive definite quadratic form to $S^{n-1}$; therefore,
$s_a$ is the support function of an ellipsoid with the center $(a_{1,1},\ldots,a_{1,n})$.

For an ellipsoid $E$, there exists a vector $a\in \mathbb{R}^{\frac{n(n+3)}{2}}$ with $a_0>-\sqrt{n}$ such that
\begin{equation}\label{characterizes ellipsoids}
h_E=\sum_{i=1}^na_{1,i}Y_{1,i}+\left(\frac{(a_0+\sqrt{n})^2}{n}+\sum_{i=1}^{N(n,2)}a_{2,i}Y_{2,i}\right)^{\frac{1}{2}};
\end{equation}
see the proof of \cite[Theorem 5.8.1]{Groemer}\footnote{Monge has found, an ellipsoid has the following property:
the vertices of all its tangential boxes (rectangular parallelepipeds) lie on a fixed sphere.
Blaschke (for $n=3$) and Chakerian (without any restriction on the dimension) have proven that this property characterizes ellipsoids among all convex bodies. Identity (\ref{characterizes ellipsoids}) is a way of formulating this statement; it says the vertices of all the tangential boxes to $E-(a_{1,1},\ldots,a_{1,n})$ lie on $(a_0+\sqrt{n})S^{n-1}.$}.

Define a map by
\begin{align*}
\xi&: \mathcal{O}_1\subset \mathbb{R}^{\frac{n(n+3)}{2}}\to  C^2(S^{n-1})\\
\xi(a):=&-1+\sum_{i=1}^na_{1,i}Y_{1,i}+\left(\frac{(a_0+\sqrt{n})^2}{n}+\sum_{i=1}^{N(n,2)}a_{2,i}Y_{2,i}\right)^{\frac{1}{2}}.
\end{align*}
The following lemma shows that the space of ellipsoids near the unit ball can be parameterized by $\xi+1.$
\begin{lemma}\label{ellip open}
There exists $0<\delta<\frac{1}{2}$ with the following property. For any ellipsoid $E$ satisfying $\|h_E-1\|_{C}< \delta$, there exists a unique $a\in \mathcal{O}_1$ such that
$\xi(a)+1=h_E.$
\end{lemma}
\begin{proof}
Take an $a$ with $a_0>-\sqrt{n}$ such that (\ref{characterizes ellipsoids}) holds. The following relations show that if $\delta$ is small enough then $a\in\mathcal{O}_1:$
\begin{enumerate}
  \item $\sum\limits_{i}|a_{1,i}|=\sum\limits_{i}|(h_E-h_{B^n},Y_{1,i})|\leq n\|h_E-1\|_2\leq n(n\omega_n)^{\frac{1}{2}}\delta,$
  \item $
\left(n\omega_n\left(\frac{(a_0+\sqrt{n})^2}{n}-1\right)^2+\sum\limits_{i}a_{2,i}^2\right)^{\frac{1}{2}}=\|h_{E-(a_{1,1},\ldots,a_{1,n})}^2-1\|_2,
$
  \item $\|h_{E-(a_{1,1},\ldots,a_{1,n})}^2-1\|_2\leq (n\omega_n)^{\frac{1}{2}}\|h_{E-(a_{1,1},\ldots,a_{1,n})}^2-1\|_C\leq c(\sum\limits_{i}|a_{1,i}|+\delta),$
  for some positive constant $c$ independent of $E.$
\end{enumerate}
The proof of the uniqueness claim is elementary.
\end{proof}
\begin{lemma}\label{lem: ellipsoids} The following statements hold.
\begin{enumerate}
  \item There exists an origin-centered open ball $\mathcal{O}_2\subset \mathcal{O}_1$ such that for any open set $\mathcal{O}\subset \mathcal{O}_2$, $(\pi_0+\pi_1+\pi_2)(\xi(\mathcal{O}))$ is an open set in $\mathcal{S}^0\oplus\mathcal{S}^1\oplus\mathcal{S}^2\equiv\mathbb{R}^{\frac{n(n+3)}{2}}.$
  \item Put $S:=\{a\in \mathbb{R}^{\frac{n(n+3)}{2}}; a_{1,i}=0, i=1,\ldots,n\}$.
  There exists an origin-centered open ball $\mathcal{O}_2\subset \mathcal{O}_1\cap S$ such that for any open set $\mathcal{O}\subset \mathcal{O}_2$, the set $(\pi_0+\pi_2)(\xi(\mathcal{O}))$ is open in $\mathcal{S}^0\oplus\mathcal{S}^2\equiv\mathbb{R}^{\frac{n(n+1)}{2}}.$
\end{enumerate}
\end{lemma}
\begin{proof}
We only prove the first claim.
The map $\xi$ is smooth and its derivative at the origin is given by
\begin{align}\label{derivative}
\partial\xi(0)&: \mathbb{R}^{\frac{n(n+3)}{2}}\to C^2(S^{n-1})\\
\partial\xi(0)h=&\sqrt{\omega_n}h_{0,1}Y_{0,1}+\sum_{i=1}^{n} h_{1,i}Y_{1,i}+\frac{1}{2}\sum_{i=1}^{N(n,2)} h_{2,i}Y_{2,i}.\nonumber
\end{align}
Let $\{\lambda_0,\lambda_1,\ldots,\lambda_{\frac{n(n+3)}{2}-1}\}$ be the coordinates of $(\pi_0+\pi_1+\pi_2)\xi(a)$ with respect to the basis $\{Y_{k,l};~ l=1,\ldots, N(n,k),~k=0,1,2\}$ of $\mathcal{S}^0\oplus \mathcal{S}^1\oplus \mathcal{S}^2 $. By (\ref{derivative}), we obtain $$\partial \lambda(0)=\operatorname{diag}(\sqrt{\omega_n},\underbrace{1,\ldots,1}_{n~\text{times}},\frac{1}{2},\ldots,\frac{1}{2}).$$
Hence the inverse function theorem implies that there exists an origin-centered open ball $\mathcal{O}_2\subset \mathcal{O}_1$ such that $\lambda:\mathcal{O}_2\to \lambda(\mathcal{O}_2)$ is a smooth diffeomorphism.
\end{proof}
The proof of the next theorem was inspired by the work of Simon \cite{Simon}.
\begin{theorem}\label{main prop}
Suppose $m\ge 4.$ There exists $\varepsilon_{m}>0$, such that if $h_K$ satisfies
$\mathcal{X}_m(h_K)=0$ and
$\|h_{\phi K}-1\|_{C^{2}}\leq \varepsilon_{m}$ for some $\phi\in \operatorname{Gl}_n,$ then $K$ is an origin-centered ellipsoid.
In particular, if $\Pi^2K=cK$ for some positive constant $c$ and $\|h_{\phi K}-1\|_{C^{2}}\leq \varepsilon_{4}$ for some $\phi\in \operatorname{Gl}_n$, then $K$ is an origin-centered ellipsoid.
\end{theorem}
\begin{proof}
Consider the map
\begin{align*}
\mathcal{N}&:\tilde{U}\subset C_e^2(S^{n-1})\to C_e^2(S^{n-1}),\quad f\mapsto\mathcal{Y}_m(f)+(\pi_0+\pi_2)f,\quad \mathcal{N}(0)=0.
\end{align*}
\begin{itemize}
  \item By Lemma \ref{cont key}, $\mathcal{N}$ is smooth.
  \item By Lemma \ref{lem 9}, the linear map
  \begin{align*}
  D\mathcal{N}(0,\cdot): C_e^2(S^{n-1})\to C_e^2(S^{n-1}),\quad f\mapsto D\mathcal{Y}_m(0,f)+(\pi_0+\pi_2)f
  \end{align*}
  is an isomorphism.
\end{itemize}
By the inverse function theorem (see, \cite[Theorem 5.2.3, Corollary 5.3.4]{Hamilton}), we can find open neighborhoods $U,~W$ of $0$ in $C_e^{2}(S^{n-1})$, such that $\mathcal{N}: U\subset \tilde{U}\to W$
is a smooth diffeomorphism. Put $M=\mathcal{N}^{-1}(W\cap \operatorname{Ker}D\mathcal{Y}_m(0,\cdot))$. A key observation is that $M$ contains all solutions of $\mathcal{Y}_m(\cdot)=0$ in $U;$ that is,
\begin{align}\label{key observ}
f\in U~\text{and}~\mathcal{Y}_m(f)=0\Rightarrow\mathcal{N}^{-1}((\pi_0+\pi_2) (f))=f\in M.
\end{align}

Define $W':=\{(\pi_0+\pi_2)(h_E-1);~ h_E-1\in U,~E=-E\}.$ By (\ref{key observ}), we have $\mathcal{N}^{-1}(W')=\{h_E-1;~h_E-1\in U,~E=-E\}.$
Also, due to Lemmas \ref{ellip open} and  \ref{lem: ellipsoids}, if $U$ is small enough then $W'$ is open in $W\cap\operatorname{Ker}D\mathcal{Y}_m(0,\cdot).$

Let us define the open set
$W_1:=\{f\in W; (\pi_0+\pi_2)f\in W'\}.$
From the definition of $W'$, it follows that $W'=W_1\cap \operatorname{Ker}D\mathcal{Y}_m(0,\cdot);$ therefore, we have
 \[\mathcal{N}^{-1}(W')=\mathcal{N}^{-1}(W_1\cap W\cap \operatorname{Ker}D\mathcal{Y}_m(0,\cdot))=\mathcal{N}^{-1}
(W_1)\cap M.\]
Note that $\mathcal{N}^{-1}
(W_1)\subset U$ contains an open neighborhood of $0$ in $C_e^2(S^{n-1}).$ Thus there exists $\varepsilon_m>0$, such that
if $\mathcal{Y}_m(f)=0$ and $\|f-1\|_{C^2}\leq \varepsilon_m$, then $$f\in M\cap \mathcal{N}^{-1}
(W_1)=\{h_E-1;~ h_E-1\in U,~E=-E\}.$$
Consequently, in a $C_e^2(S^{n-1})$-neighborhood of $1$ the only solutions of $\mathcal{Y}_m(\cdot-1)=0$ are support functions of origin-centered ellipsoids.

To prove the last statement, observe that if $\Pi^2K=cK$ for some positive constant $c$, then $\Pi^8\phi K=c'\phi K$ for some positive constant $c'.$ So $\mathcal{X}_4(h_{\phi K})=\mathcal{Y}_4(h_{\phi K}-1)=0.$ Hence $K$ must be an origin-centered ellipsoid.
\end{proof}
\begin{remark}\label{keyrem} Here $c_i$ are positive constants depending only on $n.$
Suppose $K$ has a positive continuous curvature function $f_K$. The inequality (\ref{second}) implies that
\[\|\mathcal{C}f_K-\mathcal{C}1\|_{C^2}=\|\mathcal{C}(f_K-1)\|_{C^2}\leq c_1\|f_K-1\|_{C}\Rightarrow \|h_{\Pi K}-h_{\Pi B^n}\|_{C^2}\leq c_2 \|f_K-1\|_{C}.\]
Therefore, if $\|f_K-1\|_{C}$ is small enough, then
\[\|f_{\Pi K}-f_{\Pi B^n}\|_{C}\leq c_3 \|f_K-1\|_{C}\Rightarrow \|h_{\Pi^2 K}-h_{\Pi^2 B^n}\|_{C^2}\leq c_4 \|f_K-1\|_{C}.\]
In particular, for any solution of $\Pi^2K=cK$ there holds
\[\|h_{\psi K}-1\|_{C^2}\leq c_5 \|f_K-1\|_{C},\]
where $\psi=\frac{c}{\omega_{n-1}^{n}}\operatorname{I}_{n\times n}$. Therefore, Theorem \ref{main prop} (or the main result of the paper) holds under the assumption that $\|f_{\phi K}-1\|_C$ is small enough.
\end{remark}
\section{equality cases of (\ref{SZ})}\label{sec 5}
For a body $K\in \mathrm{K}^n$ of class $C^2_+$ define
\[\mathcal{L}h_K:=\Box h_K-(n-1)^2\mathcal{R}^2\Box h_K+\frac{(n-2)(n-1)}{\omega_n}\int\limits_{S^{n-1}} h_Kdx.\]
We will show that
\begin{align*}
&W_{n-2}(\Pi_{n-2}K)-\frac{n(n-2)\omega_{n-1}^2}{(n-1)^2\omega_n}W^2_{n-1}(K)-\frac{\omega_{n-1}^2}{(n-1)^2}W_{n-2}(K)\\
&=-c_n\int\limits_{S^{n-1}} h_K\mathcal{L}h_Kdx=c_n\sum_{k=3}^{\infty}(k-1)(n+k-1)(1-v_{k,n}^2(n-1)^2)\|\pi_kh_K\|^2_2.
\end{align*}
Here, $c_n$ is some positive constant independent of $K.$

By \cite[p. 1109]{GG}, for any convex body $L$ with $h_L\in C^2(S^{n-1})$ we have
\begin{align}\label{1}
W_{n-2}(L):=V(B^n,\ldots,B^n, L,L)=\frac{1}{n(n-1)}\int\limits_{S^{n-1}}h_L\Box h_Ldx.
\end{align}
In addition, for a convex body $L$ of class $C^2_+$ we have
\begin{align}\label{2}
h_{\Pi_{n-2}L}(u)&:=h_{\Pi(B^n,\ldots,B^n,L)}(u)\\
&=\frac{1}{2}\int\limits_{S^{n-1}}|u\cdot x|dS(B^n,\ldots,B^n,L)\nonumber\\
&=\frac{1}{2(n-1)}\mathcal{C}\Box h_L(u).\nonumber
\end{align}
Since $\mathcal{C}\Box h_K\in C^2(S^{n-1}),$ putting (\ref{1}) and (\ref{2}) together yields
\[W_{n-2}(\Pi_{n-2}K)=\frac{1}{4n(n-1)^3}\int\limits_{S^{n-1}}\mathcal{C}\Box h_K\Box\mathcal{C}\Box h_Kdx.\]
We use that both $\mathcal{C}$ and $\Box$ are self-adjoint to obtain
\begin{align*}
W_{n-2}(\Pi_{n-2}K)&=\frac{1}{4n(n-1)^3}\int\limits_{S^{n-1}}\mathcal{C}\Box h_K\Box\mathcal{C}\Box h_Kdx\\
&=\frac{1}{4n(n-1)^3}\int\limits_{S^{n-1}} h_K\Box\mathcal{C}\Box\mathcal{C}\Box h_Kdx.
\end{align*}
Thus by the identity (\ref{C and R}) we arrive at the formula
\begin{align}\label{3}
W_{n-2}(\Pi_{n-2}K)=\frac{\omega_{n-1}^2}{n(n-1)}\int\limits_{S^{n-1}}h_K\mathcal{R}^2\Box h_K dx.
\end{align}
We also need
\begin{align}\label{4}
W_{n-1}(K):=V_1(B^n,K)=\frac{1}{n}\int\limits_{S^{n-1}}h_Kdx.
\end{align}
Combining (\ref{1}), (\ref{3}) and (\ref{4}) all together gives
\begin{align*}
&W_{n-2}(\Pi_{n-2}K)-\frac{n(n-2)\omega_{n-1}^2}{(n-1)^2\omega_n}W^2_{n-1}(K)-\frac{\omega_{n-1}^2}{(n-1)^2}W_{n-2}(K)\\
&=\frac{\omega_{n-1}^2}{n(n-1)}\int\limits_{S^{n-1}}h_K\mathcal{R}^2\Box h_K dx-\frac{(n-2)\omega_{n-1}^2}{n(n-1)^2\omega_n}\left(\int\limits_{S^{n-1}}h_Kdx\right)^2-\frac{\omega_{n-1}^2}{n(n-1)^3}\int\limits_{S^{n-1}}h_K\Box h_{K}dx\\
&=\frac{\omega_{n-1}^2}{n(n-1)^3}\int\limits_{S^{n-1}}h_K\left((n-1)^2\mathcal{R}^2\Box h_K+\frac{(2-n)(n-1)}{\omega_n}\int\limits_{S^{n-1}}h_Kdx-\Box h_K\right)dx\\
&=-c_n\int\limits_{S^{n-1}}h_K\mathcal{L}h_Kdx.
\end{align*}
Note that $\mathcal{L}1=0$. Assume $Y_k$ is a spherical Harmonic of degree $k\geq 1$. Since $\int_{S^{n-1}}Y_kdx=0,$ from $\Box Y_k=-(k-1)(n+k-1)Y_k$ and $\mathcal{R}Y_k=(-1)^{k/2}v_{k,n}Y_k$ it follows that
\[\mathcal{L}Y_k=-(k-1)(n+k-1)(1-(n-1)^2v_{k,n}^2)Y_k.\]
Therefore, since $\mathcal{L}$ is linear and self-adjoint, if $K$ is of class $C^2_+$ then
\begin{align*}
(h_K,\mathcal{L}h_K)&=-\int\limits_{S^{n-1}}\sum_{k= 0}^{\infty} \pi_kh_K\left(\sum_{k=3}^{\infty} (k-1)(n+k-1)(1-v_{k,n}^2(n-1)^2)\pi_kh_K\right)dx\\
&=-\sum_{k=3}^{\infty}(k-1)(n+k-1)(1-v_{k,n}^2(n-1)^2)\|\pi_kh_K\|_2^2,
\end{align*}
and consequently for any $k\ge 3$ we obtain
\[W_{n-2}(\Pi_{n-2}K)-\frac{n(n-2)\omega_{n-1}^2}{(n-1)^2\omega_n}W^2_{n-1}(K)-\frac{\omega_{n-1}^2}{(n-1)^2}W_{n-2}(K)\geq c_{n,k}\|\pi_kh_K\|_2^2\geq 0.\]
The class of $C^{\infty}_+$ convex bodies is dense in $\mathrm{K}^n$. Thus a standard approximation shows that this last inequality holds for all convex bodies. Thus the left-hand side is zero precisely when the support function has the form
\[h_K=(\pi_0+\pi_1+\pi_2)h_K.\]
Note that, indeed, we have proved
\[\frac{d^2}{dt^2}\Big|_{t=0}\mathcal{P}(1+th_K)=\frac{d^2}{dt^2}\Big|_{t=0}\frac{V(\Pi(1+th_K))}{V(1+th_K)^{n-1}}=-c_n'\int\limits_{S^{n-1}}h_K\mathcal{L}h_Kdx\ge 0.\]
and $\operatorname{Ker}\mathcal{L}=\mathcal{S}^0\oplus\mathcal{S}^1\oplus\mathcal{S}^2.$ The $\operatorname{Ker}\mathcal{L}$ reflects the symmetries (the affine invariance) that $\mathcal{P}$ enjoys; $\mathcal{S}^0$, $\mathcal{S}^1$, and $\mathcal{S}^2$ are present in $\operatorname{Ker}\mathcal{L}$, respectively, because $\mathcal{P}$ is scaling-invariant, translation-invariant, and $\operatorname{Sl}_n$-invariant.
\bibliographystyle{amsplain}

\end{document}